\documentclass[11pt,twoside]{amsart}
\usepackage{latexsym,amssymb,amsmath}

\numberwithin{equation}{section}
\newtheorem{theorem}{Theorem}[section]
\newtheorem{lemma}[theorem]{Lemma}
\newtheorem{proposition}[theorem]{Proposition}
\newtheorem{corollary}[theorem]{Corollary}

\theoremstyle{definition}
\newtheorem{definition}[theorem]{Definition} 
\newtheorem{remark}[theorem]{Remark}

\begin{document}


\newcommand{\m}[1]{\marginpar{\addtolength{\baselineskip}{-3pt}{\footnotesize
\it #1}}}
\newcommand{\A}{\mathcal{A}}
\newcommand{\Z}{\mathbb{Z}}
\newcommand{\N}{\mathbb{N}}
\newcommand{\R}{\mathbb{R}}


\title{Toric ideals generated by circuits}

\author{Jos\'e Mart\'{\i}nez-Bernal}
\address{
Departamento de
Matem\'aticas\\
Centro de Investigaci\'on y de Estudios Avanzados del
IPN\\
Apartado Postal
14--740 \\
07000 Mexico City, D.F. } \email{jmb@math.cinvestav.mx}

\author{Rafael H. Villarreal}
\address{
Departamento de
Matem\'aticas\\
Centro de Investigaci\'on y de Estudios
Avanzados del
IPN\\
Apartado Postal
14--740 \\
07000 Mexico City, D.F.
}
\email{vila@math.cinvestav.mx}
\urladdr{http://www.math.cinvestav.mx/$\sim$vila/}
\thanks{The second author was partially supported by CONACyT 
grant 49251-F and SNI}

\begin{abstract} Let $I_\A$ be the toric ideal of a homogeneous normal
configuration $\mathcal{A}\subset\mathbb{Z}^n$. We prove that $I_\A$ 
is generated by circuits if and only if each unbalanced circuit of
$I_\A$ has a {\it connector\/} which is a linear combination of circuits with a
square-free term. In particular, if each circuit of $I_\A$ 
with non-square-free terms is balanced, then $I_\A$ is generated by
circuits. As a consequence we
prove that the toric ideal of a normal edge subring of a multigraph is
generated by circuits with a square-free term.
\end{abstract}

\maketitle

\section{Introduction}

Let $\A$ be a finite {\it vector configuration\/} in $\mathbb{R}^n$
and let $I_\A$ be its 
associated {\it toric ideal\/}, i.e., $\A=\{v_1,\ldots,v_q\}\subset\Z^n$ and
$I_\A$ is the prime ideal of $K[T]$ given by (see \cite{Stur1}):
$$I_\A=\left(
T^a-T^b\vert\, a=(a_i),b=(b_i)\in\N^q,\textstyle\sum_ia_iv_i=\sum_i
b_iv_i\right)\subset K[T],$$ 
where $K[T]=K[T_1,\ldots,T_q]$ is a polynomial ring over a field $K$.
Here we will 
use  $T^a$ as
an abbreviation for  $T_1^{a_1} \cdots T_q^{a_q}$, 
where $a=(a_i)$ is a vector in 
$\mathbb{N}^q$. Let $A$ be the matrix with column vectors
$v_1,\ldots,v_q$ and let $V$ be the kernel of $A$ in $\mathbb{Q}^q$.
An integral vector $0\neq \alpha\in V$ is called a {\it circuit} of $V$ 
if $\alpha$ has minimal support and its non-zero entries are
relatively prime. A binomial $T^a-T^b$ is called a {\it circuit} of $I_\A$ 
if $a-b$ is a circuit of $V$ and $\gcd\{T^a, T^b\}=1$.  
In Section~\ref{section-toric} the
notion of a circuit is discussed in more detail. 
The number of circuits of $V$ is 
finite \cite[Corollary~8.4.5]{monalg}
and the circuits of $V$
generate $V\cap\mathbb{Z}^n$ as a group \cite[Theorem~8.4.10]{monalg}. 
The notion of a {\it circuit\/} occurs in convex analysis 
\cite{Rock}, in the theory of toric ideals of graphs 
\cite{BGS,Stur1,Vi3}, and in matroid theory \cite{oxley}. The ideal
generated by the circuits of $I_\A$ has been studied in \cite{BJT}. 
We are interested in configurations whose toric ideal $I_\A$ is
generated by circuits. The best known examples of toric ideals
generated by circuits come from configurations whose matrix $A$ is
unimodular \cite{Stur1}. In Corollary~\ref{another-corollary} we
present a larger family of toric ideals generated by circuits. 
Another interesting family of examples of toric ideals generated by 
circuits are the phylogenetic ideals studied in \cite{sonja}. As noted
in \cite{sonja1}, these phylogenetic ideals actually represent the
family of cut ideals of cycles. Toric ideals have been widely studied
from various points of view and have interesting connections 
with commutative algebra, geometry, integer programming and graph
theory \cite{BMT,Ful1,Stur1,monalg}.

The configuration $\A$
is {\it normal} if $\N\A=\Z\A\cap\mathbb{R}_+\A$, and {\it
homogeneous\/} if $\A$ lies on an affine hyperplane 
in $\R^n$ not containing the origin. As usual, 
$\mathbb{N}\mathcal{A}$ (resp. $\mathbb{Z}\mathcal{A}$) denotes the
semigroup (resp. subgroup) of $\mathbb{Z}^n$ generated by $\mathcal{A}$, and
$\mathbb{R}_+\mathcal{A}$ denotes the cone generated by 
$\mathcal{A}$ consisting of the linear combinations of $\mathcal{A}$
with coefficients in $\mathbb{R}_+=\{x\in\mathbb{R}\vert\,x\geq 0\}$.
A binomial $T^a-T^b$ has a {\it square-free term} if $T^a$ is square-free or 
$T^b$ is square-free.
The main result of this paper is Theorem~\ref{maintheorem}. It shows
that the toric ideal $I_\A$ of a homogeneous normal
 configuration is generated by circuits if and only if
every unbalanced  
circuit of $I_\A$ has a {\it connector\/} which is a
$K[T]$-linear combination of circuits of $I_\A$ with a square-free term. 
See Definition~\ref{balanced} for a precise definition of the notion
of a balanced 
circuit as well as for that of a connector. In particular, if
$\mathcal{A}$ is homogeneous, normal and each circuit of $I_\A$ 
with non-square-free terms is balanced, then $I_\A$ is generated by
circuits (see Corollary~\ref{main-corollary}).  
As an application we prove that normal edge ideals of 
multigraphs are generated by circuits with a square-free term (see
Theorem~\ref{toric-gen-by-circuits}).

\section{Circuits of Toric ideals}\label{section-toric}
Let $\mathcal{A}=\{v_1,\ldots,v_q\}\subset\mathbb{Z}^n$ be a finite vector
configuration, let $A$ be the integral matrix with column vectors
$v_1,\ldots,v_q$, and let $I_\A$ be the associated toric ideal of
$\mathcal{A}$. For $\alpha=(\alpha_i)\in {\mathbb R}^q$, its {\it
support\/} is defined as
$${\rm supp}(\alpha)=\{i\, |\, \alpha_i\neq 0\}.$$
Note that $\alpha=\alpha_+-\alpha_-$, 
where $\alpha_+$ and $\alpha_-$ are two non-negative vectors 
with disjoint support. 

\begin{definition}\rm Let $V=\{\alpha\vert\, A\alpha=0\}$ be the kernel
of $A$ in $\mathbb{Q}^q$. A {\it circuit\/} of $A$ is a 
non-zero integral vector 
$\alpha$ in $V$ whose support is minimal with
respect to inclusion (i.e., ${\rm supp}(\alpha)$ does not properly
contains the support of any other non-zero vector in $V$) and such 
that the non-zero entries of $\alpha$ are relatively prime. A circuit
of $A$ is also called a circuit of $V$. If
$\alpha$ is a circuit of $A$, we call the binomial
$T^{\alpha_+}-T^{\alpha_-}$ a {\it circuit\/} of the toric ideal $I_\A$.
\end{definition}

Let $E=\{1,\ldots,q\}$ be the set of column labels of the matrix $A$
and let $\mathcal I$ be 
the set of 
subsets $B$ of $E$ for which the
multiset of columns labeled by $B$ is linearly independent in
$\mathbb{Q}^n$. Then $M[A]=(E,{\mathcal I})$ is a {\it matroid\/} on $E$ 
by \cite[Proposition~1.1.1, p.~8]{oxley}, i.e., the collection
$\mathcal I$ satisfies the following three conditions:
\begin{enumerate}
\item[$(\mathrm{i}_1)$] $\emptyset\in \mathcal I$.

\item[$(\mathrm{i}_2)$] If $I\in\mathcal I$ and $I'\subset I$, 
then $I'\in
\mathcal I$.

\item[$(\mathrm{i}_3)$] If $I_1$ and $I_2$ are in $\mathcal I$ and 
$|I_1|<|I_2|$, then there is an element $e$ of $I_2\setminus I_1$ such
that $I_1\cup\{e\}\in \mathcal I$.
\end{enumerate} 

The matroid $M[A]$ is called the {\it vector matroid\/} of $A$ over
the field $\mathbb{Q}$. The reader is referred to \cite{oxley} for a 
general theory of matroids. The members of $\mathcal I$ are the {\it
independent} 
sets of $M[A]$.
A subset of 
$E$ that is not in $\mathcal I$ is called {\it
dependent\/}.  A minimal dependent set 
in $M[A]$ is called a {\it circuit\/} of $M[A]$. 
It is interesting to observe that there is a correspondence
$$
\begin{array}{ccc}
\mbox{Circuits of } A&\longrightarrow & \mbox{ Circuits of }
M[A]\\
\alpha=(\alpha_1,\ldots,\alpha_q)&\longmapsto &
{\rm supp}(\alpha).
\end{array}
$$
Thus the set of circuits of $A$ is an algebraic
realization of the set of circuits of the vector matroid 
$M[A]$.

\begin{definition}\rm Two vectors 
$\alpha=(\alpha_i)$ and $\beta=(\beta_i)$ 
in ${\mathbb Q}^q$ are in {\it harmony\/} 
if $\alpha_i\beta_i\geq 0$ for all $i$. 
\end{definition}

\begin{lemma}[\cite{Rock}]\label{harmonyl} Let $V$ be the kernel of
$A$ in $\mathbb{Q}^q$. 
If $0\neq\alpha\in V$, then there is a circuit $\gamma\in V$ 
in harmony with $\alpha$ such that 
${\rm supp}(\gamma)\subset {\rm supp}(\alpha)$.  
\end{lemma}

A non-zero binomial $T^a-T^b$ is said to have a {\it square-free term} if
$a_i\in\{0,1\}$ for all $i$ or $b_i\in\{0,1\}$ for all $i$. If 
$a_i\notin\{0,1\}$ for some $i$ and $b_j\notin\{0,1\}$ for some $j$,
we say that the binomial $T^a-T^b$ has {\it non-square-free terms}.

\begin{lemma}{\cite[Proposition~4.1]{birational}}\label{square-free-term}
If $\A$ is homogeneous, normal, and $I_\A$ is minimally generated by a
finite set $\mathcal{B}$ consisting of binomials, then every element
of $\mathcal{B}$ has a square-free term.
\end{lemma}

\begin{definition}\label{balanced} 
A binomial $g=T_1^{a_1}\cdots T_q^{a_q}-T_{1}^{b_1}\cdots T_q^{b_q}$ 
is called {\it balanced} if the following holds:
$$
\max\{a_1,\ldots,a_q\}=\max\{b_1,\ldots,b_q\}.
$$ 
If $g$ is not balanced it is called {\it unbalanced}. Let $g$ be an
unbalanced binomial of the form:
\[g=T_{1}^{b_1}\cdots
T_{r}^{b_r}-T_{r+1}^{b_{r+1}}\cdots T_{s}^{b_s},\ \ \ b_i\geq
1\,\forall i,\]
where $1\leq
m_1=\max\{b_1,\ldots,b_r\}<\max\{b_{r+1},\ldots,b_s\}=m_2$. A {\it
connector} of $g$ is a binomial: 
\[T_{i_1}\cdots T_{i_j}-T_{i_{j+1}}^{c_{j+1}}\cdots T_{i_m}^{c_m}
,\ \ \ c_i\geq
1\,\forall i,\]
with a square-free term $T_{i_1}\cdots T_{i_j}$ such that
$\{i_1,\ldots,i_j\}\subset\{1,\ldots,r\}$ and the intersection of 
$\{i_{j+1},\ldots,i_m\}$ with $\{r+1,\ldots,s\}$ is non-empty. 
\end{definition}

We come to the main result of this paper.

\begin{theorem}\label{maintheorem} Let $\A$ be a homogeneous normal
configuration and let $I_\A$ be its toric ideal. The following are
equivalent{\rm :} 
\begin{enumerate}
\item[\rm (a)] $I_\A$ is generated by a finite set of circuits.
\item[\rm (b)] $I_\A$ is generated by a finite set of 
circuits with a square-free term.
\item[\rm (c)] Every unbalanced circuit of $I_\A$ has a connector which is a
linear combination {\rm(}with coefficients in $K[T]${\rm)} of
circuits of $I_\A$ with a square-free term.  
\end{enumerate}
\end{theorem}

\begin{proof} Let $R=k[x_1^{\pm1},\ldots,x_n^{\pm1}]$ be the ring 
of Laurent polynomials over a field $K$ 
and let $K[F]$ be the subring
of $R$ generated by $F=\{x^{v_1},\ldots,x^{v_q}\}$. Since $\mathcal{A}$ is
normal, we get that $K[F]$ is normal, i.e., 
$\overline{K[F]}=K[F]$, where $\overline{K[F]}$ 
is the integral closure of $K[F]$ in its field of fractions. As
$\mathcal{A}$ is homogeneous, we get that any binomial $T^a-T^b$ 
in $I_\A$ is homogeneous with respect to the standard grading of
$K[T]=K[T_1,\ldots,T_q]$ induced by setting $\deg(T_i)=1$ for all $i$, 
a fact that will be used repeatedly below without any
further notice.  

(a) $\Rightarrow$ (b): The toric ideal $I_\A$ is minimally generated by 
a finite set $\mathcal{B}$ of circuits. Thus, by
Lemma~\ref{square-free-term}, each binomial of $\mathcal{B}$ is a
circuit with a square-free term. 

$(b) \Rightarrow (c)$: Let $g$ be an unbalanced circuit of $A$:
\[g=T_{1}^{b_1}\cdots T_{r}^{b_r}-T_{r+1}^{b_{r+1}}\cdots
T_{s}^{b_s},\ \ \ b_i\geq 1\,\forall i,\] 
where $1\leq
m_1=\max\{b_1,\ldots,b_r\}<\max\{b_{r+1},\ldots,b_s\}=m_2$. We may
assume $m_2=b_{r+1}$. Then 
\begin{equation}\label{nov30-08}
(x^{v_1}\cdots x^{v_r}/x^{v_{r+1}})^{m_1}\in K[F].
\end{equation}
The element $x^{v_1}\cdots x^{v_r}/x^{v_{r+1}}$ is in the field of
fractions of $K[F]$ and by Eq.~(\ref{nov30-08}) it is integral over
$K[F]$. Hence, by the normality of $K[F]$, the element 
$x^{v_1}\cdots x^{v_r}/x^{v_{r+1}}$ is in $K[F]$. Since $K[F]$ is
generated as a $K$-vector space by Laurent monomials of the form
$x^a$, with $a\in\mathbb{N}\mathcal{A}$, it is not hard to see that
there is a monomial $T^\gamma$ such that
\[T_{1}\cdots T_r-T_{r+1}T^\gamma\in I_\A.\]
This is a connector of $g$ and by hypothesis it is a linear combination
of circuits of $I_\A$ with a square-free term. 

(c) $\Rightarrow$ (a): By Lemma~\ref{square-free-term}, the toric ideal $I_\A$
is minimally generated by a finite set 
$$\mathcal{B}=\{f_1,\ldots,f_m\}$$
consisting of 
binomials with a square-free term. We will show, by induction on the
degree, that each one of the $f_i$'s is a linear
combination of circuits. The degree is taken with respect to the
standard grading of $K[T]$.

Let $f$ be a binomial in $\mathcal{B}$. 
We may assume that $f$
has the form:
\[f=T_1\cdots T_p-T_{p+1}^{a_{p+1}}\cdots T_\ell^{a_\ell} ,\ \ \ 
a_i\geq 1\,\forall i, \ \ell\leq q.\]
Assume that $f$ is not a circuit. Then by Lemma~\ref{harmonyl} 
there is a circuit in $I_\A$ (permuting variables if
necessary) of the form
\[g=T_1^{b_1}\cdots T_r^{b_r}-T_{p+1}^{b_{p+1}}\cdots T_s^{b_s},\ \ \
b_i\geq 1\, \forall i,\]
with $r<p$ or $s<\ell$. We set $m_1=\max\{b_1,\ldots,b_r\}$ and
$m_2=\max\{b_{p+1},\ldots,b_s\}$. 

We claim that there exist binomials $h$ and $h_1$ (we allow $h=h_1$
or $h=0$) 
in $I_\A$  of degree less than $\deg(f)=p$ and a binomial $h_2$ which is a
linear combination of circuits of $I_\A$ such that $f$ is in the
ideal of $K[T]$ generated by $g,h,h_1,h_2$. To prove
this we consider the following two cases.

\underline{Case} (A): $r=p$ and $s<\ell$. Then
\begin{equation}\label{referee1}
g=T_1^{b_1}\cdots T_p^{b_p}-T_{p+1}^{b_{p+1}}\cdots T_s^{b_s}.
\end{equation}

\underline{Subcase} ($\mathrm{A}_1$): $b_i=1$ for $i=1,\ldots,p$. Then we can
write 
\[f-g=T_{p+1}^{b_{p+1}}\cdots T_s^{b_s}-T_{p+1}^{a_{p+1}}\cdots
T_\ell^{a_\ell}=T_{p+1}h,\]
for some binomial $0\neq h\in I_\A$ (recall that $I_\A$ is a prime
ideal) with $\deg(h)<\deg(f)=p$.

\underline{Subcase} ($\mathrm{A}_2$): $b_i=1$ for $i=p+1,\ldots,s$. Then
\[g=T_1^{b_1}\cdots T_p^{b_p}-T_{p+1}\cdots T_s.\]
By subcase ($\mathrm{A}_1$), we may assume that $b_i\geq 2$ for some
$1\leq i\leq p$. Then, on the one hand, by the homogeneity of $g$, 
$p+1\leq\sum_{i=1}^p b_i=s-p$, so $2p+1\leq s$. On the other hand,
by the homogeneity of $f$, $p\geq \ell-p\geq s-p+1$, so $2p-1\geq s$.
This is a contradiction. So this case cannot occur.

\underline{Subcase} ($\mathrm{A}_3$): $b_i\geq 2$ for some $1\leq
i\leq p$, $b_{p+j}\geq
2$ for some $1\leq j\leq s-p$, 
and $m_1\geq m_2$. Then
\[f=T_1\cdots T_p-T_{p+1}^{a_{p+1}}\cdots T_\ell^{a_\ell},\]
\[g=T_1^{b_1}\cdots T_p^{b_p}-T_{p+1}^{b_{p+1}}\cdots T_s^{b_s}.\]
For simplicity of notation we may
assume that $m_1=b_1$. Using that $g\in I_\A$ and $m_1\geq m_2\geq 2$, we
get 
$$
(x^{v_{p+1}}\cdots x^{v_s}/x^{v_1})^{m_2}\in
K[F].
$$
Hence, by the normality of $K[F]$, there is a
monomial $T^\gamma$ such that the binomial  
\[h_1=T_{p+1}\cdots T_s-T_1T^\gamma\]
is in $I_\A$. The binomial $h_1$ is non-zero and has degree 
less than $\deg(f)$ because $s-p<\ell-p\leq p$; 
the second inequality follows from the homogeneity
of $f$. 
Let
\[T^\delta=T_{p+1}^{a_{p+1}}\cdots T_\ell^{a_\ell}/T_{p+1}\cdots T_s.\]
We have
\begin{equation}\label{referee2}
f+h_1T^\delta=f+(T_{p+1}\cdots T_s-T_1T^\gamma)T^\delta=T_1\cdots
T_p-T_1T^\gamma T^\delta 
=T_1h,
\end{equation}
where $0\neq h\in I_\A$ and $\deg(h)<p$. 

\underline{Subcase} ($\mathrm{A}_4$): $b_i\geq 2$ for some $1\leq
i\leq p$, $b_{p+j}\geq
2$ for some $1\leq j\leq s-p$, and $m_1<m_2$. Since $g$ is an unbalanced circuit, by
hypothesis 
$g$ has a connector 
\[h_2=T_{i_1}\cdots T_{i_k}-T_{i_{k+1}}T^\gamma,\ \ i_1<\cdots<i_k,\]
with $i_{k+1}\in\{p+1,\ldots,s\},\,
\{i_1,\ldots,i_k\}\subset\{1,\ldots,p\}$ and such that $h_2$ is a linear
combination of 
circuits of $I_\A$. Set
\[T^\delta=T_1\cdots T_p/T_{i_1}\cdots T_{i_k}.\]
If $f=h_2T^\delta$, then $T^\delta=1$ and $f=h_2$. 
If $f\neq h_2T^\delta$, then 
we can write
\begin{equation}\label{referee3}
f-h_2T^\delta=T_{i_{k+1}}T^\gamma T^\delta-T_{p+1}^{a_{p+1}}\cdots
T_\ell^{a_\ell}=T_{i_{k+1}}h,
\end{equation}
with $0\neq h\in I_\A$ and $\deg(h)<p$.

\underline{Case} (B): $r<p$ and $s\leq \ell$. In this case

\[f=T_1\cdots T_p-T_{p+1}^{a_{p+1}}\cdots T_\ell^{a_\ell},\]
\[g=T_1^{b_1}\cdots T_r^{b_r}-T_{p+1}^{b_{p+1}}\cdots T_s^{b_s}.\]

\underline{Subcase} ($\mathrm{B}_1$):  $b_i=1$ for $i=1,\ldots,r$.
Then
\begin{equation}\label{referee4}
f-gT_{r+1}\cdots T_p=T_{p+1}^{b_{p+1}}\cdots T_s^{b_s}T_{r+1}\cdots
T_p-T_{p+1}^{a_{p+1}}\cdots T_\ell^{a_\ell}=T_{p+1}h,
\end{equation} 
where $0\neq h\in I_\A$ and $\deg(h)<p$.

\underline{Subcase} ($\mathrm{B}_2$):  $b_i=1$ for $i=p+1,\ldots,s$. Let
\[T^\gamma=T_{p+1}^{a_{p+1}}\cdots T_\ell^{a_\ell}/T_{p+1}\cdots
T_s.\]
Then we have
\begin{equation}\label{referee5}
f-gT^\gamma=T_1\cdots T_p-T_1^{b_1}\cdots T_r^{b_r}T^\gamma=T_1h
\end{equation}
where $0\neq h\in I_\A$ and $\deg(h)<p$.

\underline{Subcase} ($\mathrm{B}_3$):  $b_i\geq 2$ for some $1\leq
i\leq r$, $b_j\geq 2$ for some $p+1\leq j\leq s$, and $m_1\leq m_2$. We may assume
$m_2=b_{p+1}$. Using that $g\in I_\A$ and $m_2\geq m_1\geq 2$, we
get 
$$
(x^{v_{1}}\cdots x^{v_r}/x^{v_{p+1}})^{m_1}\in
K[F].
$$
Hence, by the normality of $K[F]$, there is a
monomial $T^\gamma$ such that the binomial  
\[h_1=T_1\cdots T_r-T_{p+1}T^\gamma\]
is in $I_\A$. The binomial $h_1$ is non-zero and has degree 
less than $\deg(f)$ because $r<p$. 
Then we have
\begin{eqnarray}
f-h_1T_{r+1}\cdots T_p&=&
f-(T_1\cdots T_r-T_{p+1}T^\gamma)T_{r+1}\cdots T_p\label{referee6}\\ 
&=&
T_{p+1}T^\gamma T_{r+1}\cdots T_p-T_{p+1}^{a_{p+1}}\cdots
T_\ell^{a_\ell}=T_{p+1}h,\nonumber
\end{eqnarray}
where $0\neq h\in I_\A$ and $\deg(h)<p$. 

\underline{Subcase} ($\mathrm{B}_4$):  $b_i\geq 2$ for some $1\leq
i\leq r$, $b_j\geq 2$ for some $p+1\leq j\leq s$, and 
$m_1>m_2$. Since $g$ is an unbalanced circuit, by hypothesis $g$ has a connector
\[h_2=T_{i_{k+1}}\cdots T_{i_{k+t}}-T_{i_{d}}T^\gamma,\ \ 
i_{k+1}<\cdots<i_{k+t},
\] 
with $\{i_{k+1},\cdots,i_{k+t}\}\subset\{p+1,\ldots,s\}$,
$i_d\in\{1,\ldots,r\}$, and such that $h_2$ is a linear
combination of circuits of $I_\A$. Set
\[T^\delta=T_{p+1}^{a_{p+1}}\cdots T_\ell^{a_\ell}/T_{i_{k+1}}\cdots
T_{i_{k+t}}.\]
If $f=-h_2T^\delta$, then $T^\delta=1$ and $f=-h_2$. If 
$f\neq-h_2T^\delta$, then we can write
\begin{equation}\label{referee7}
f+h_2T^\delta=T_1\cdots T_p-T_{i_d}T^\gamma T^\delta=T_{i_{d}}h,
\end{equation}
with $0\neq h\in I_\A$ and $\deg(h)<p$. This completes the proof of
the claim.

We are now ready to show that each $f_i$ in $\mathcal{B}$ is a linear
combination of circuits. We proceed by induction on $\deg(f_i)$. 
Let $p=\min\{\deg(f_i)\vert\, 1\leq i\leq m\}$ be the initial degree
of $I_\A$. If $f_i$ is a binomial in $\mathcal{B}$ of degree $p$, 
then either $f_i$ is a circuit or $f_i$ is not a circuit and by the
claim $f_i$ is a linear
combination of circuits (notice that in this case $h=h_1=0$ 
because there are no non-zero binomials in $I_\A$ of degree less than
$p$). Let 
$d$ be an integer 
greater than
$p$ and let $f_k$ 
be a binomial of $\mathcal{B}$ of degree $d$ (if any). Assume that each $f_i$ of 
degree less than $d$ is a linear combination of circuits. If $f_k$ is a circuit
there is nothing to prove. If $f_k$ is not a circuit, then by the
claim (or more precisely by Eqs.
(\ref{referee1})--(\ref{referee7})) we
can write
\begin{equation}\label{oct27-08}
f_k={\lambda}g+{\mu}h+\mu_1h_1+\mu_2h_2, 
\end{equation}
where $\lambda,\mu,\mu_1,\mu_2$ are monomials, 
$h,h_1$ are binomials in
$I_\A$ of
degree less than $d=\deg(f_k)$, $h_2$ is a linear combination of
circuits, and $g$ is a circuit. Since $I_\A$ is a graded ideal with
respect to the standard grading of $K[T_1,\ldots,T_q]$, we get that 
$h$ and $h_1$ are linear combinations of binomials in $\mathcal{B}$ of
degree less than $d$. Therefore by Eq.~(\ref{oct27-08}) and the induction
hypothesis, we conclude that $f_k$ is a linear combination of
circuits. Therefore the ideal $I_\A$ is generated by a 
finite set of circuits. 
\end{proof}

\begin{remark}\label{referee-remark} In the proof of 
Theorem~\ref{maintheorem} (from (c) to (a)), the subcases 
($\mathrm{A}_3$),  ($\mathrm{B}_1$), and ($\mathrm{B}_3$) cannot
occur. Indeed, since $f_i$
belongs to a minimal system of binomial generators of the toric
ideal it cannot be written as a linear combination of binomials of the
toric ideal of degree strictly smaller. In 
Eq.~(\ref{oct27-08}) either $\lambda$ or $\mu_2$ has to be nonzero.  
\end{remark}

The following result will be used in Section~\ref{multigraphs} to
show a class of toric ideals generated by circuits. 

\begin{corollary}\label{main-corollary} 
Let $\A$ be a homogeneous normal
configuration and let $I_\A$ be its toric ideal. If each circuit 
of $I_\A$ with non-square-free terms is balanced, then 
$I_\A$ is generated by a finite set of circuits with a square-free term. 
\end{corollary}

\begin{proof} The circuits of $I_\A$ satisfy condition (c) of 
Theorem~\ref{maintheorem}. Indeed, let $f$ be an unbalanced circuit
of $I_\A$. Then $f$ has a 
square-free term by hypothesis. Thus $f$ is a circuit with a
square-free term and it is a connector of $f$. Hence the result
follows 
from Theorem~\ref{maintheorem}.
\end{proof}

\begin{corollary}\label{another-corollary} 
Let $\A\subset\mathbb{N}^n\setminus\{0\}$ be a homogeneous
configuration 
and let $I_\A$ be its toric ideal. If each circuit 
of $I_\A$ has a square-free term, then $\mathcal{A}$ is normal and 
$I_\A$ is generated by a finite set of circuits with a square-free term. 
\end{corollary}

\begin{proof} The normality of $\A$ follows from
\cite[Theorem~2.3]{unimod}. Since the circuits of $I_\A$ satisfy condition (c) of 
Theorem~\ref{maintheorem}, we get that $I_\A$ is generated by a 
finite set of circuits with a square-free term. 
\end{proof}

\section{Toric ideals of normal edge subrings}\label{multigraphs}
Let $G$ a multigraph with vertex set $X=\{x_1,\ldots,x_n\}$, i.e., $G$
is obtained from a simple graph by allowing multiple edges and
multiple loops.
Thus the edges of $G$ have the form $\{x_i,x_j\}$. If $e=\{x_i,x_j\}$
is an edge of $G$, its characteristic vector is given by
$v_e=e_i+e_j$, where $e_i$ is the $i${\it th} unit vector in
$\mathbb{R}^n$. Notice that if $e$ is a loop, i.e., if $i=j$, 
then $v_e=2e_i$. The {\it incidence matrix} of $G$, denoted by $A$, 
is the matrix whose
column vectors are the characteristic vectors of the edges and loops of $G$.
Since we are allowing multiple edges some of the columns of $A$ may be
repeated. Let $v_1,\ldots,v_q$ be the characteristic vectors of
the edges and loops of $G$ and let $\mathcal{A}=\{v_1,\ldots,v_q\}$ be its
associated vector configuration. 

Let $K[T_1,\ldots,T_q]$ be a polynomial ring 
over a field $K$. The {\it edge subring} of $G$ is the
monomial subring:
$$
K[G]=K[x^{v_1},\ldots,x^{v_q}]\subset K[x_1,\ldots,x_n],
$$
where $K[x_1,\ldots,x_n]$ is a polynomial ring with coefficients in
$K$. 
It is well known that the toric ideal $I_\mathcal{A}$ is the kernel of
the following epimorphism of $K$-algebras
$$
K[T_1,\ldots,T_q]\longrightarrow K[G]
$$
induced by $T_i\mapsto x^{v_i}$. 

\begin{proposition}\label{constraints} If $f=T^a-T^b$ is a circuit of the toric ideal 
$I_\mathcal{A}$, then $f$ has a square-free term or $f$ has
non-square-free terms and $\max_i\{a_i\}=\max_i\{b_i\}=2$.
\end{proposition}

\begin{proof} If $G$ is a simple graph, the result was shown in
\cite[Corollary~4.1]{Vi3}. The general case, i.e., the multigraph
case, follows using an identical argument as the one given in \cite{Vi3}.
\end{proof}

We come to the main application of this paper.

\begin{theorem}\label{toric-gen-by-circuits} Let $G$ be a multigraph
and let $I_\A$ be the toric 
ideal of the edge subring $K[G]$. Then $K[G]$ is normal if and only if
$I_\A$ is generated by circuits with a square-free term.
\end{theorem}

\begin{proof} $\Rightarrow$) This direction follows at once applying
Proposition~\ref{constraints} and Corollary~\ref{main-corollary}.

$\Leftarrow$) It is seen using a description of the
integral closure of the edge subring $K[G]$ given in \cite{bowtie} (cf.
\cite[Proposition~5.9]{BGS}).
\end{proof}

\begin{definition} A sub-multigraph $H$ of $G$ is called a {\it
circuit\/} of $G$ if $H$ has one 
of the following forms:
\begin{enumerate}
\item[\rm (a)] $H$ is an even cycle. 

\item[\rm (b)] $H$ consists of two odd cycles intersecting in
exactly one vertex; a loop is regarded as an odd cycle of length $1$. 

\item[\rm (c)] $H$ consists of two vertex disjoint odd cycles joined
by a path.
\end{enumerate}
\end{definition}

The circuits of $G$ are in one to one correspondence 
with the circuits of $I_\A$ as we now explain, see \cite{Vi3} for a
detailed discussion. Any circuit $H$ of $G$ can be regarded as an even
closed walk
$$
w=\{w_0,w_1,\ldots,w_r,w_0\},
$$
where $r$ is even, $w_0,w_1,\ldots,w_r$ are the vertices of $H$ (we
allow repetitions) and $\{w_i,w_{i+1}\}$ is an edge or loop of $G$
for all $i$. Then the binomial $T_w=T_1T_3\cdots T_{r-1}-T_2T_4\cdots
T_r$ is in $I_\A$, where $f_i=w_{i-1}w_i$ and $T_i$ maps to $f_i$ for
all $i$. 

\begin{remark} The circuits of $I_\A$ with a
square-free term correspond to the following types of circuits of $G$:

\begin{enumerate}
\item[\rm (a)] Even cycles. 

\item[\rm (b)] Two odd cycles intersecting in
exactly one vertex.

\item[\rm (c)] Two vertex disjoint odd cycles joined by an edge.
\end{enumerate}
Thus if $K[G]$ is a normal subring, by Theorem~\ref{toric-gen-by-circuits}
we obtain a very precise
graph theoretical description of a generating set of circuits for $I_\A$. 
\end{remark}

Toric ideals of edge subrings of oriented graphs were studied in
\cite{GRV,ringraphs}. 
In this case the toric ideal is also generated by circuits
and the circuits correspond to the cycles of the graph. 

\bigskip

\noindent
{\bf Acknowledgments.} We thank the referees for a
careful reading of the paper and for the improvements suggested. 

\medskip

\noindent {\it Note added in proof\/}: Apostolos Thoma has pointed out to us that the
condition of Theorem 3.2 is not sufficient for the normality of
the edge subring. A sufficient condition is the one given
in Corollary 2.9.

\bibliographystyle{plain}

\end{document}